\documentclass[12pt,oneside]{amsart}

\usepackage[hyperindex]{hyperref}

\usepackage{cite}

\usepackage{latexsym}
\usepackage[utf8]{inputenc}
\usepackage{amssymb,amsfonts,amsbsy,amsmath,amsopn,amsthm}
\usepackage{geometry}
\usepackage{enumerate}
\usepackage{color}

\newtheorem{theorem}{Theorem}%[section]
\newtheorem{lemma}[theorem]{Lemma}

\newtheorem{proposition}[theorem]{Proposition}
\newtheorem{corollary}[theorem]{Corollary}

\theoremstyle{definition}
\newtheorem{remark}[theorem]{Remark}

\def\sign{\mathrm{sign}}

\title[Lacunary Walsh series in r.i.\ spaces]
{
Lacunary Walsh series in rearrangement invariant spaces
}

\author{Javier Carrillo--Alan\'is}

\address{Departamento de Análisis Matemático, 
	Facultad de Matemáticas,
	Universidad de Sevilla.
	C/ Tarfia S/N.
    41012 Sevilla, Spain.}
\email{fcarrillo@us.es}

\thanks{Partially supported by MTM2015-65888-C4-1-P}

\begin{document}

\begin{abstract} 
We prove that the classical results by Rodin and Semenov and by Lindenstrauss 
and Tzafriri
on the subspace generated by the Rademacher system in  rearrangement invariant spaces
also hold for  lacunary Walsh series.
\end{abstract}

\maketitle 

%%%%%%%%%%%%%%%%%%%%%%%%%%%%%%%%%%%%%%%%%%%%%%

\section{Introduction}

The Rademacher functions are 
$$
r_k(t):= \sign \sin(2^k \pi t), \qquad t \in [0,1], \quad k \geq 1.
$$
The Rademacher system $(r_k)$ is orthonormal, not complete, independent and identically distributed 
on $[0,1]$. Its behavior in rearrangement invariant function spaces has been pretty well studied. 
Khintchine inequality, \cite{MR1544623}, and the theorems by Rodin and Semenov 
and  Lindenstrauss and Tzafriri
are among the core results
on this topic.

\begin{theorem}[Khintchine inequality]
Given $0<p<\infty$, there exist constants $A_p, B_p >0$ such that 
$$
A_p \Big( \sum_{k \geq 1} a_k^2  \Big)^{1/2} \leq \Big\| \sum_{k \geq 1} a_k r_k \Big\|_{L^p([0,1])}
\leq 
B_p \Big( \sum_{k \geq 1} a_k^2  \Big)^{1/2},
$$
for all $(a_k) \in \ell^2$. 
\end{theorem}

Rodin and Semenov characterized those rearrangement invariant spaces $X$
for which Khintchine inequality holds with $X$ in place of $L^p$, \cite[Theorem 6]{MR0388068}. Denote by $G$ 
the closure in $L^\infty$ of the Orlicz space $L^{M_2}$ generated by $M_2(t):=\exp(t^2)-1$.
 \begin{theorem}[Rodin and Semenov]\label{th:Rodin-Semenov}
  Let $X$ be a rearrangement invariant space on $[0,1]$. The following conditions are equivalent:
  \begin{enumerate}[i)]
   \item There exist constants $A_X, B_X >0$ such that the inequality
   $$
A_X \Big( \sum_{k \geq 1} a_k^2  \Big)^{1/2} \leq \Big\| \sum_{k \geq 1} a_k r_k \Big\|_{X}
\leq 
B_X \Big( \sum_{k \geq 1} a_k^2  \Big)^{1/2}
$$
holds for all $(a_k) \in \ell^2$.
\item The continuous embedding $G \subset X$ holds.
  \end{enumerate}
 \end{theorem}

The complementability of the closed linear subspace $[r_k]_X$  generated by the Rade\-macher system 
in a rearrangement invariant space $X$ was also characterized by means of the space $G$, \cite[Theorem 2.b.4]{MR540367} and
\cite{MR541648}. Denote 
by $X'$ the associate space of $X$.
 \begin{theorem}[Rodin and Semenov, Lindenstrauss and Tzafriri]\label{th:complementability}
  Let $X$ be a rearrangement invariant space on $[0,1]$. The following conditions are equivalent:
  \begin{enumerate}[i)]
   \item The subspace $[r_k]_X$ is complemented in $X$.
\item The continuous embeddings $G \subset X$ and $G \subset X'$ hold.
  \end{enumerate}
 \end{theorem}

The Walsh system $(w_k)$ on $[0,1]$ consists of all finite products 
of Rademacher functions. Unlike the Rademacher system, the Walsh system is complete and not independent.
However, Sagher and Zhou showed that Khintchine inequality also holds for lacunary sequences of Walsh functions, \cite[Theorem 1]{MR1052010}.

\begin{theorem}[Sagher and Zhou]\label{th:sagherzhou4}
  Given $0<p<\infty$ and $q>1$, there exist constants $A(p,q), B(p,q)>0$ such that, 
  for any sequence $(w_{n_k})$ of Walsh functions with $n_{k+1}/n_k \geq q >1$ for all $k \geq 1$, the inequalities
  $$
  A(p,q)\Big( \sum_{k=1}^\infty a_k^2  \Big)^{1/2}
  \leq \Big\|\sum_{k =1}^\infty a_k w_{n_k}  \Big\|_{L^p([0,1])}
  \leq B(p,q)\Big( \sum_{k=1}^\infty a_k^2  \Big)^{1/2}
  $$
  hold for all $(a_k)_1^\infty \in \ell^2$.
\end{theorem}

In this paper we show that Theorem \ref{th:Rodin-Semenov} and Theorem \ref{th:complementability}
also hold for
lacunary sequences of Walsh functions (Theorem \ref{th:rodin-semenov-walsh} and Theorem \ref{th:complemented-walsh}). 

We also consider \textit{local} versions of Khintchine inequality.
The first \textit{local} result was given by Zygmund for $L^2$, \cite[Lemma V.8.3]{MR0617944}.
\begin{lemma}[Zygmund] \label{lemma:zygmund-local}
There exist constants $A'_2, B'_2 >0$ such that, for any set $E \subset [0,1]$ with $m(E)>0$,
there exists $N=N(E)$ such that 
$$
A'_2 \Big( \sum_{k \geq N} a_k^2  \Big)^{1/2} 
\leq \Big(\int_E \Big|\sum_{k \geq N} a_k r_k \Big |^2 \frac{dm}{m(E)}\Big)^{1/2}
\leq 
B'_2 \Big( \sum_{k \geq N} a_k^2  \Big)^{1/2},
$$
for all $(a_k)$ in $\ell^2$.
\end{lemma}

Zygmund's local result has been generalized in two different directions:
first, by considering rearrangement invariant function spaces $X$ with $G \subset X \subset L^2$ 
 (see, for example,  
\cite{MR3372735}, \cite{MR3345658},\cite{MR2801616}, \cite{MR3764631}, \cite{MR1044796} and \cite{MR1412621}); 
and second, by
considering lacunary sequences of Walsh functions in place of the Rademacher system.
In this regard, Sagher and Zhou proved the following result, \cite[Theorem 2]{MR1052010}.

\begin{theorem}[Sagher and Zhou]\label{th:sagherzhou5}
  Given $0<p<\infty$ and $q>1$, there exist constants $A'(p,q),B'(p,q)>0$ such that 
  for any  set $E \subset [0,1]$ of positive measure,
  there exists $N=N(E,q)$ so that
for any sequence $(w_{n_k})$ of Walsh functions with $n_{k+1}/n_k \geq q >1$ for all $k \geq 1$, the inequalities
  $$
  A'(p,q)\Big( \sum_{k=1}^\infty a_k^2  \Big)^{1/2} \leq
\Big(\int_E \Big|\sum_{k \geq N} a_k w_{n_k} \Big |^p \frac{dm}{m(E)}\Big)^{1/p}  
\leq B'(p,q)\Big( \sum_{k=1}^\infty a_k^2  \Big)^{1/2}
  $$
 hold for all $(a_k)_1^\infty \in \ell^2$.
\end{theorem}

We extend this local result for lacunary sequences of Walsh functions on $L^p$
to the space $L^{M_2}$ of functions of square exponential integrability (Theorem \ref{th:local-walsh-LM2}).

\section{Preliminaries}

A Banach function space over $[0,1]$ is a  linear subspace $X$ of  measurable functions on 
$[0,1]$, endowed with a complete norm $\|\cdot\|_X$, such that $g \in X$ 
and $|f| \leq |g|$ a.e.\ implies $f \in X$ and $\|f\|_X \leq \|g\|_X$.
The associate space $X'$ of a Banach function space $X$ consists of all measurable functions $g$ on $[0,1]$ 
for which the associate functional
 $$
 \|g\|_{X'}:=\sup \Big\{  \Big| \int_0^1 f g \, dm \, \Big| : \|f\|_X \leq 1 \Big\}
 $$
is finite.
The inclusion $X' \subset X^*$ always holds between the associate space $X'$ and the dual
Banach  space  $X^*$. 
The spaces $X'$ and $X^*$ are isomorphic if and only if $X$ 
has absolutely continuous norm (that is, order bounded 
increasing sequences are norm convergent).

We  denote the distribution function of a measurable function $f$ on $[0,1]$ by
$m_f (\lambda):= m(\{ x \in [0,1] : |f(x)| > \lambda \})$, for all $ \lambda >0$.
A Banach function space $X$ over $[0,1]$ is \emph{rearrangement invariant} (r.i.)
if $m_f = m_g$ and $f \in X$ imply $g \in X$ and $\|g\|_X = \|f\|_X$.
The associate space $X'$ of an r.i.\ space $X$ is an r.i.\ space. 

The second associate space of $X$ is $X'':=(X')'$. The embedding 
$X \subset X''$ holds for any Banach function space $X$. 
A Banach function space $X$ satisfies the Fatou property if $f_n \in X$ with $\|f_n\|_X \leq M$ for all $n \geq 1$ and 
$0 \leq f_n \leq f_{n+1}\nearrow f$ a.e.\ implies that $f \in X$ and 
$\|f\|_X = \sup_n \|f_n\|_X$. 
 A Banach function space  $X$ satisfies the Fatou property 
if and only if $X''$ coincides with $X$.

For $X$ an r.i.\ space on $[0,1]$, the embeddings $L^\infty \subset X \subset L^1$ hold.
Denote by $X_0$ the closure of $L^\infty $ in $X$.

A Banach space $X$ on $[0,1]$ 
for which the continuous embeddings $L^\infty \subset X \subset L^1$ hold
is called an intermediate space between $L^1$ and $L^\infty$.
The space $X$ is called an \textit{interpolation space} if, 
for every linear operator $T$ such that $T \colon L^1\to L^1$
and $ T\colon L^\infty \to L^\infty$ are continuous, 
then $T\colon X \to X$ is continuous.

Interpolation spaces between $L^1([0,1])$ and $L^\infty([0,1])$ are
(after renorming if necessary) rearrangement invariant,
and rearrangement invariant spaces which satisfy the Fatou property 
or are separable are interpolation spaces between $L^1([0,1])$ and $L^\infty([0,1])$
(for precise details, see \cite[Chp.\ II, \S4]{MR649411}).

A Young function is as function of the form
$$
\Phi(s)=\int_0^s \phi(t) \, dt,
$$
where $\phi \colon [0,+\infty) \to [0,+\infty)$ is increasing, left--continuous and $\phi(0)=0$.
The Orlicz space $L^\Phi$ generated by a Young function $\Phi$ 
consists of all measurable functions  $f$ on $[0,1]$ for which the norm
$$
\|f\|_{L^\Phi}:=\inf\Big\{ \lambda >0 : \int_0^1 \Phi(|f|/\lambda) \,dm \leq 1  \Big\}
$$
is finite. The space $G:=(L^{M_2})_0$, where
 $L^{M_2}$ is the Orlicz space generated by  $M_2(t):=\exp(t^2)-1$,
is of particular interest in the study of the Rademacher system.

We consider the Walsh system according to \textit{Paley's numbering}, that is, $w_0:=1$,
and for $k=a_1 2^0 + \ldots + a_{n} 2^{n-1}$ with $a_0, \ldots, a_n \in \{0,1\}$,
$$
w_k:=r_1^{a_1} \cdot \ldots \cdot r_n ^{a_n}.
$$
A sequence $(w_{n_k})$ of the Walsh system is $q$--lacunary if $n_{k+1}/n_k \geq q >1$ for all $k \geq 1$.
Since $w_{2^n} = r_{n+1}$, the Rademacher system $(r_k)$ is a $2$--lacunary sequence of Walsh functions.

For details on the theory of rearrangement invariant spaces, see \cite{MR928802}, \cite{MR649411} and \cite{MR540367}.

\section{When is the subspace $[w_{n_k}]_X$ isomorphic to $\ell^2$?}

We start proving an extension of the local version of Khintchine inequality for
lacunary sequences of Walsh functions on
$L^p$ by Sagher and Zhou (Theorem \ref{th:sagherzhou5}).
To this aim we introduce the following concept.

Given a set $E \subset [0,1]$ with $m(E)>0$, consider the mapping 
$\rho_E \colon E \to [0,1]$ defined as $\rho_E(x):=m(E \cap [0,x])/m(E)$, $x \in [0,1]$.
There exist sets $A_1 \subset E$ and $A_2 \subset [0,1]$ of measure zero such that 
$\rho_E \colon E \setminus A_1 \to [0,1] \setminus A_2$ is bijective.
For $X$ an r.i.\ space on $[0,1]$,
the local space $X|E$ consists of all measurable functions $f$ on $E$ for which the norm 
$$
\|f\|_{X|E}:=\|f \circ \rho_E^{-1}\|_X
$$
is finite. 
Here, $\rho_E^{-1}$ is the left continuous inverse of the increasing function $\rho_E$.
The space $X|E$ is an r.i.\ space on $E$
endowed with the measure 
$$
m_E(A):=\frac{m(E \cap A)}{m(E)}, \quad  A \subset E, \quad  A \text{ measurable}.
$$
The  spaces
$L^p|E$ and $L^{M_p}|E$ coincide with the spaces in the local results by 
Zygmund for $L^2$ (Lemma \ref{lemma:zygmund-local}), 
Sagher and Zhou for $L^p$ and $L^{M_1}$ (see \cite{MR1044796}, \cite{MR1052010} and \cite{MR1412621}) 
and Carrillo--Alanís for $L^{M_2}$ \cite[Theorem 4]{MR2801616}, 
that is, for $0 < p < \infty$,
$$
\|f\|_{L^p|E} = \Big(\int_E |f|^p \frac{dm}{m(E)} \Big)^{1/p},
$$
and, for $M_p(t):=\exp (t^p) -1$, $p>0$,
$$
\|f\|_{L^{M_p}|E}:=\inf\Big\{ \lambda >0 : \int_E M_p(|f|/\lambda) \,\frac{dm}{m(E)} \leq 1  \Big\}.
$$
The definition of $X|E$ allows us to consider a local space on $E$
in the cases when an explicit expression of the norm of $X$ is not available 
(see \cite{MR3764631} for details). 

Then, we have the following result for $L^{M_2}|E$.

\begin{theorem}\label{th:local-walsh-LM2}
Let  $q>1$  and $E \subset [0,1]$ be a set of positive measure. 
Consider the space $L^{M_2}$ (of functions of square exponential integrability).
There exist constants $A'(M_2,q)$, $B'(M_2,q)>0$ 
and $N=N(E)$ such that, 
   for any $q$--lacunary sequence $(w_{n_k})$ of Walsh functions
   with $n_k \geq N$ for all $k \geq 1$, we have
$$
A'(M_2,q) \Big( \sum_{k=1}^\infty a_k^2  \Big)^{1/2}
\leq \Big\| \sum_{k \geq 1} a_k w_{n_k} \Big\|_{L^{M_2}|E} 
\leq B'(M_2,q) \Big( \sum_{k=1}^\infty a_k^2  \Big)^{1/2},$$
for all $(a_k)_1^\infty \in \ell^2$.
\end{theorem}
\begin{proof} 
We will use the inequality 
\begin{equation} \label{eq:1} B'(2n,q) \leq (1+\sqrt{2}) (2+2\alpha)^{1/2n} n^{1/2}\end{equation}
for the constant $ B'(p,q)$ in Theorem \ref{th:sagherzhou5},
where $\alpha$ is the least integer such that $q^\alpha \geq 2$ when $1 < q < 2$, and $\alpha=0$
when $q \geq 2$.
This inequality is not explicitly stated in 
\cite{MR1052010}, but it follows from the proof of Theorem \ref{th:sagherzhou5}.

Let $N=N(E)$ be as in Theorem \ref{th:sagherzhou5}.
The left--hand side inequality follows from the embedding $ L^{M_2} \subset L^1$, which implies
$ L^{M_2}|E \subset L^1|E $,
and from Theorem \ref{th:sagherzhou5} for $p=1$.
Thus, for some constant $C_1>0$,
$$C_1  {A'(1,q)} \Big( \sum_{k=1}^\infty a_k^2  \Big)^{1/2}
  \leq C_1 \Big\| \sum_{k \geq 1} a_k w_{n_k} \Big\|_{L^1|E} 
  \leq \Big\| \sum_{k \geq 1} a_k w_{n_k} \Big\|_{L^{M_2}|E},
$$
for any $q$--lacunary sequence $(w_{n_k})$ with $n_k \geq N$ and 
$(a_k)_1^\infty \in \ell^2$.

In order to prove 
the right--hand side inequality, let
$$
f:=\sum_{k \geq 1} a_k w_{n_k}.
$$
We proceed as in the proof of Theorem 4 of \cite{MR2801616},
replacing \cite[Lemma 5]{MR2801616} by \eqref{eq:1}.
From the power series expansion of $\exp(t^2)-1$ and Theorem \ref{th:sagherzhou5}
for $p=2n$, we have
\begin{equation*}\begin{split}
\int_E \big(\exp|f(t) / \lambda|^2 -1  \big)  \frac{dt}{m(E)} 
&
=\sum_{n \geq 1} \frac{1}{n!\lambda^{2n}}
\int_E \Big| \sum_{k \geq 1} a_k w_{n_k}(t) \Big|^{2n} \frac{dt}{m(E)} \\
&
\leq \sum_{n \geq 1} \frac{B'(2n,q)^{2n}}{n!\lambda^{2n}}\|(a_k)_1^\infty\|_2^{2n}.
\end{split}\end{equation*}
It follows, applying Stirling's formula, that for some absolute constant $C_2>0$,
\begin{equation*}\begin{split}
& \int_E\big( \exp|f(t) / \lambda|^2 -1  \big)  \frac{dt}{m(E)} \\
& \qquad \leq 
\sum_{n \geq 1} \frac{\big((1+\sqrt{2}) (2+2\alpha)^{1/2n} n^{1/2}\big)^{2n}}{n!\lambda^{2n}} \|(a_k)_1^\infty\|_2^{2n} \\
& \qquad =(2+2\alpha)\sum_{n \geq 1} \frac{(1+\sqrt{2})^{2n} n^n}{n!\lambda^{2n}} \|(a_k)_1^\infty\|_2^{2n}\\
& \qquad  \leq C_2 (2+2\alpha)\sum_{n \geq 1}
\Big(
\frac{(1+\sqrt{2})^2 e }{\lambda^2} \|(a_k)_1^\infty\|_2^2
\Big)^n.
\end{split}\end{equation*}
From this inequality it follows, as in the proof of Theorem 4 of \cite{MR2801616},
that there exists a constant $B'(M_2,q)>0$ such that
$$
\Big\| \sum_{k \geq 1} a_k w_{n_k} \Big\|_{L^{M_2}|E} 
\leq B'(M_2,q)\Big( \sum_{k=1}^\infty a_k^2  \Big)^{1/2},$$
and so the proof is complete.
\end{proof}

\begin{remark} 
		(i) Note that  $\alpha$ establishes the dependence between $B'(M_2,q)$ and $q$ in Theorem \ref{th:local-walsh-LM2}.
		In particular, since $q \geq 2$ implies $\alpha=0$, 
		the constant $B'(M_2,q)$ is the same for all $q \geq 2$. 

		(ii)  The non--local case of Theorem \ref{th:local-walsh-LM2}, i.e.\ for $E=[0,1]$, is referred to in \cite[p.\ 247]{MR2914604}.
\end{remark}

The next result is a consequence of Theorem \ref{th:local-walsh-LM2}.

\begin{corollary}\label{cor:1}
Let  $q>1$  and $E \subset [0,1]$ be a set of positive measure. 
Given an r.i.\ space $X$ with $G \subset X$, there exist constants
$A'(X,q)$, $B'(X,q)>0$ 
and $N=N(E)$ such that, 
   for any $q$--lacunary sequence $(w_{n_k})$ of Walsh functions
   with $n_k \geq N$ for all $k \geq 1$, we have
$$
A'(X,q) \Big( \sum_{k=1}^\infty a_k^2  \Big)^{1/2}
\leq \Big\| \sum_{k \geq 1} a_k w_{n_k} \Big\|_{X|E} 
\leq B'(X,q) \Big( \sum_{k=1}^\infty a_k^2  \Big)^{1/2},$$
for all $(a_k)_1^\infty \in \ell^2$.
\end{corollary}
\begin{proof}
 It follows from Theorem \ref{th:local-walsh-LM2}, Theorem \ref{th:sagherzhou5} for $p=1$,
 and from the fact that $G \subset X \subset L^1$ implies $G|E \subset X|E \subset L^1|E$.
\end{proof}

Note that, in the particular case when $X=L^p$, Corollary \ref{cor:1} 
coincides with the local result for lacunary Walsh series in Theorem \ref{th:sagherzhou5}.
On the other hand, since $r_{n+1}=w_{2^n}$,  Corollary \ref{cor:1}  also allows us to recover the local results 
for the Rademacher system in
 \cite{MR2801616}, \cite{MR3764631}, \cite{MR1044796} and \cite{MR1412621}.

Our next aim is to extend Theorem \ref{th:Rodin-Semenov} to lacunary series 
of Walsh functions. The next technical result is needed. The dyadic intervals of order $n$ are
$I^n_k:=(k/2^n,(k+1)/2^n)$, for $n \geq 0$ and $0 \leq k \leq 2^n-1$.

\begin{lemma} \label{lemma:1}
 Let $q \geq 2$ and  $(w_{n_k})$ be a $q$--lacunary sequence of Walsh functions.
 Then, for any $M \geq 1$ and $(a_k) \in \ell^2$, the functions
 $$R_M:= \sum_{k =1}^M a_k r_k \quad \text{and} \quad  W_M:=\sum_{k =1}^M a_k w_{n_k}$$
 have the same distribution function.
\end{lemma}
\begin{proof}
 We proceed by induction on $M$. For $M=1$, we have  $|W_1|=|R_1|$.

  Since $W_M$ is a finite sum of Walsh functions, 
 it is constant on the dyadic intervals of order $N$,
 where $N$ is such that $ 2^{N-1} \leq n_M < 2^N$. 
 For $\varepsilon_i=\pm 1$, consider the set $A(\varepsilon_1 ,\ldots ,\varepsilon_M)$,
 where $w_M$ takes the value $a_1 \varepsilon_1 + \ldots +	 a_M \varepsilon_M$.
  Since $A(\varepsilon_1 ,\ldots ,\varepsilon_M)$ consists of a finite union of dyadic intervals
  of order less or equal than $2^N$, we have 
     $$W_M = \sum_{\varepsilon_1 ,\ldots ,\varepsilon_M = \pm1} ( \varepsilon_1 a_1 + \ldots \varepsilon_M a_M )
  \chi_{A(\varepsilon_1 ,\ldots ,\varepsilon_M)},$$ 
  with $m(A(\varepsilon_1 ,\ldots ,\varepsilon_M))=1/2^M$.
  
  Assume that $W_M$ and $R_M$ have the same distribution function. From the fact that $n_{M+1}/n_M  \geq 2$
 it follows that $n_{M+1} \geq 2^N$, and so the order of $w_{M+1}$ is
  greater or equal than $N+1$. It follows that  each set $A(\varepsilon_1 ,\ldots ,\varepsilon_M)$
 can be divided  into two sets, $A(\varepsilon_1 ,\ldots ,\varepsilon_M,1)$ and $A(\varepsilon_1 ,\ldots ,\varepsilon_M,-1)$,
 consisting on finite unions of dyadic intervals of order less or equal than $N+1$,
 both of the same measure, where $w_{M+1}$ takes values $1$ and $-1$. Thus,
  $W_{M+1}$ has the same distribution function that $R_{M+1}$.
\end{proof}

We prove a version of Theorem \ref{th:Rodin-Semenov}
for lacunary sequences of Walsh functions.

\begin{theorem} \label{th:rodin-semenov-walsh}
Let $X$ be an r.i.\ space on $[0,1]$.  The following conditions are equivalent.
\begin{enumerate}[(i)]
 \item The continuous embedding $G \subset X$ holds, that is, 
there exists a constant $C>0$ such that
$$
\|f\|_X \leq C \|f\|_{L^{M_2}}
$$
for all $f \in L^\infty$.
 \item For any $q >1$, there exist  constants $A(X,q) ,B(X,q)>0$ such that 
 $$ 
 A(X,q) \Big( \sum_{k=1}^\infty a_k^2  \Big)^{1/2} 
 \leq \Big \|\sum_{k \geq 1} a_k w_{n_k} \Big \|_X
 \leq B(X,q) \Big( \sum_{k=1}^\infty a_k^2  \Big)^{1/2} , 
 $$
  for all  $(a_k)_1^\infty \in \ell^2$, and any $q$--lacunary system $(w_{n_k})$ of Walsh functions.
 \item There exist a sequence $(n_k)$ and   constants
$A(X), B(X)>0$ such that
$$A(X) \Big( \sum_{k=1}^\infty a_k^2  \Big)^{1/2} \leq \Big\| \sum_{k \geq 1} a_k w_{n_k} \Big\|_X \leq 
B(X)
\Big( \sum_{k=1}^\infty a_k^2  \Big)^{1/2},$$
for all $(a_k)_1^\infty \in \ell^2$. 
\end{enumerate}

\end{theorem}

\begin{proof} $\emph{(i)} \Rightarrow \emph{(ii)}$
Assume that $G \subset X$. Since for any r.i.\ space $X$ the continuous embedding $X \subset L^1$ holds,
then \emph{(ii)} follows from Theorem \ref{th:sagherzhou4} for $p=1$
and from Theorem \ref{th:local-walsh-LM2} with $E=[0,1]$. 

$ \emph{(ii)}\Rightarrow \emph{(iii)}$ Is clear.

$\emph{(iii)} \Rightarrow \emph{(i)}$ 
To show $\emph{(i)}$  it suffices to assume that the right--hand side
inequality in $\emph{(iii)}$ holds, that is, for some constant $B(X)>0$,
$$
\Big\| \sum_{k \geq 1} a_k w_{n_k} \Big\|_X \leq B(X)
\Big( \sum_{k=1}^\infty a_k^2  \Big)^{1/2},
$$
for any $(a_k)_1^\infty \in \ell^2$.
Consider a subsequence $(m_k) \subset (n_k)$ such that
$m_{k+1}/m_k \geq 2$ for all $k \geq 1$, 
and let 
$$
s_n:=\frac{1}{\sqrt{n}}\sum_{k=1}^n w_{m_k}, \qquad v_n:=\frac{1}{\sqrt{n}}\sum_{k=1}^n r_k.
$$
Then, from Lemma \ref{lemma:1}, $s_n$ and $v_n$ have the same distribution function, for all $n \geq 1$.
From \emph{(iii)}, we have  $\|s_n\|_X \leq B(X)$. Thus,
$$
\|v_n\|_X = \|s_n\|_X \leq  B(X),
$$
and so $v_n \in X$, and $v_n$ are uniformly bounded in norm. 
Following the steps of the proof  of Theorem \ref{th:Rodin-Semenov}  by Rodin and Semenov 
(see \cite[Theorem 6]{MR0388068}),
$v_n \in X$ with $\|v_n\|_X \leq  B(X)$ for all $n \geq 1$ implies,  via the Central Limit Theorem, that $G \subset X$.
\end{proof}

\begin{remark} Let $(w_{n_k})$ be a $q$--lacunary sequence of Walsh functions, $X$ an 
r.i.\ space on $[0,1]$ and $(a_k) \in \ell^2$ with $\sum_{k \geq 1} a_k r_k \in X$.

If $q \geq 2$, then it follows from  Lemma \ref{lemma:1} that
$$
\Big \|\sum_{k \geq 1} a_k w_{n_k} \Big \|_X
=
\Big \|\sum_{k \geq 1} a_k r_k \Big \|_X.
$$
Combining this fact together with the results mentioned in Introduction, we get at once 
Theorems \ref{th:local-walsh-LM2} and  \ref{th:rodin-semenov-walsh}  in the case when $q \geq 2$.

Consider now the case $1 < q < 2$. From \cite[Theorem 8.1(d)]{MR2525624} and Theorem \ref{th:sagherzhou5} we have 
that $(w_{n_k})$ is majorized in distribution 
by $(r_k)$, that is, there exists a constant $C > 1$ such that
$$
m\Big( \Big\{ t \in [0,1] : \Big| \sum_{k = 1}^M a_k w_{n_k}(t)  \Big| > \lambda \Big\} \Big) 
\leq 
C m\Big(\Big\{ t \in [0,1] : \Big| \sum_{k = 1}^M a_k w_{n_k}(t)  \Big| >  \frac{\lambda}{C} \Big\}\Big),
$$
for all $\lambda >0$, $M \in \mathbb{N}$ and  $a_1, \ldots, a_M \in \mathbb{R}$.
From the boundedness on any r.i.\ space $X$ of 
the dilation operator $\sigma_{1/C}$, 
$$
(\sigma_{1/C}f)(t):= f(t/C), \qquad f \in X, \quad 0<t<C,
$$  
with norm $\| \sigma_{1/C} \|_X \leq C$,
it follows that 
$$
\Big \|\sum_{k \geq 1} a_k w_{n_k} \Big \|_X
\leq C^2
\Big \|\sum_{k \geq 1} a_k r_k \Big \|_X.
$$
Note that this inequality suffices in order to prove 
Theorem \ref{th:local-walsh-LM2} for $E=[0,1]$, and it also 
holds even in the case when $G \nsubseteq X$.
The opposite majoration, that is, 
$(r_k)$ being majorized in distribution by $(w_{n_k})$,
is an open problem for which we
have not found any references. 
It is a relevant question in the context of this paper,
since it would imply that lacunary Walsh series and 
Rademacher series have equivalent norms in any r.i.\ space.
\end{remark}

\section{Complementability} 

The main result of this section is the following.

\begin{theorem}\label{th:complemented-walsh}
Let $X$ be an r.i.\ space on $[0,1]$ which is an
interpolation space between $L^1([0,1])$ and $L^\infty([0,1])$. 
 The following conditions are equivalent.
\begin{enumerate}[(i)]
 \item The continuous embeddings $G \subset X$ and $G \subset X'$ hold, that is,
there exist constants $C, C'>0$ such that 
$$
\|f\|_X \leq C \|f\|_{L^{M_2}}, \qquad \|f\|_{X'} \leq C' \|f\|_{L^{M_2}},
$$
for all $f \in L^\infty$.
\item For any $q >1$ and any $q$--lacunary sequence $(w_{n_k})$ of Walsh functions,
the space $[w_{n_k}]_X$ is complemented in $X$.
\item There exists $q>1$ and a  $q$--lacunary sequence $(w_{n_k})$ of Walsh functions such that
$[w_{n_k}]_X$  is complemented in $X$.
\end{enumerate}
\end{theorem}

The proof follows the ideas of \cite[Theorem 2.b.4]{MR540367}
and \cite{MR541648}. We need some auxiliary results.

\begin{proposition}\label{prop:basic_walsh} 
Let $X$ be an r.i.\ space on $[0,1]$ which is an
interpolation space between $L^1([0,1])$ and $L^\infty([0,1])$,
and $(w_{n_k})$ a $q$--lacunary subsequence of the Walsh system with $q >1$.
Then, $(w_{n_k})$ is a basic sequence in $X$.
\end{proposition}
\begin{proof} The case when $q \geq 2$ follows from Lemma \ref{lemma:1} and from the fact that $(r_k)$
	is a basic sequence in $X$.
	
Let $1 < q < 2$.
For $s \in \mathbb{N}$, denote by $I^s_j = (j/2^s,(j+1)/2^s)$  the dyadic intervals
of order $s$,  $0 \leq j \leq 2^s-1$,
and consider the averaging operator 
$$
A_s f:=\sum_{j=0}^{2^s-1} \Big( \frac{1}{m(I^s_j)} \int_{I^s_j} f \, dm \Big) \chi_{I^s_j}, 
\qquad
f \in L^1([0,1]).    
$$
The operators $A_s \colon X \to X$ are uniformly bounded (see \cite[II,\S3.2]{MR649411}).

Let $0\leq k < 2^s$. Since $w_k$ is constant on 
 $\chi_{I^s_j}$ for $0\leq j \leq 2^s-1$, we have
$$
A_s ( w_k)=w_k, \qquad 0\leq k < 2^s.
$$
On the other hand, noting that for $k \geq 2^s$ and $ 0 \leq j \leq  2^s-1$,
$$
\int_{I^s_j} w_k \,dm=0,
$$  
it follows that  $A_s( w_k)=0$, $k \geq 2^s$.

Let $M,N \in \mathbb{N}$ with $M<N$. 
Note, for $1 \leq j \leq N$, that
\begin{equation} \label{eq:a_j}
|a_j| = \Big| \int_0^1 \Big(\sum_{k=1}^N a_k w_{n_k}\Big)  w_{n_j} \,dm \Big| 
\leq \Big\|\sum_{k=1}^N a_k w_{n_k}\Big\|_{L^1}
\leq C_1 \Big\| \sum_{k=1}^N a_k w_{n_k}\Big\|_X,
\end{equation}
  where $C_1$ denotes the constant in the continuous embedding $X \subset L^1$.

There are two cases. Suppose first that there exists $s \in \mathbb{N}$ such that
$2^{s-1} \leq n_M < 2^s \leq n_N$. Let $L$ such that $n_L < 2^s \leq n_{L+1}$, with  $M \leq L < L+1 \leq N$. Then,
$$
\sum_{k=1}^M a_k w_{n_k} = A_s\Big(\sum_{k=1}^N a_k w_{n_k} \Big) - \sum_{k=M+1}^L a_k w_{n_k}.
$$
Note that the last sum vanishes  in the case when $M=L$.
It follows that
$$
\Big\|\sum_{k=1}^M a_k w_{n_k}\Big\|_X \leq \Big\| A_s\Big(\sum_{k=1}^N a_k w_{n_k}\Big)\Big\|_X + \sum_{k=M+1}^L |a_k| \varphi_X(1),
$$
where $\varphi_X(t):=\| \chi_{[0,t]}\|_X$, $0 \leq t \leq 1$, denotes the fundamental function of the space $X$.

Suppose next that there exists $s \in \mathbb{N}$ such that
$2^s \leq  n_M <  n_N < 2^{s+1} $ (in this case necessarily $1 < q < 2$). Let $L$ such that 
$n_L < 2^s \leq n_{L+1}$, with  $L+1 \leq M$. Then,
$$
\sum_{k=1}^M a_k w_{n_k} = A_s\Big(\sum_{k=1}^N a_k w_{n_k}\Big) + \sum_{k=L+1}^M a_k w_{n_k},
$$
and so we have
$$
\Big\|\sum_{k=1}^M a_k w_{n_k}\Big\|_X \leq \Big\| \Big(A_s\sum_{k=1}^N a_k w_{n_k}\Big)\Big\|_X + \sum_{k=L+1}^M |a_k| \varphi_X(1).
$$

Let $\alpha \in \mathbb{N}$ be such that $q^\alpha \geq 2$. 
From the lacunary condition, we have in both cases that $|L-M| \leq \alpha$. 
Taking \eqref{eq:a_j} into account,
it follows that 
\begin{equation*}\begin{split}
\Big\| \sum_{k=1}^M a_k w_{n_k}\Big\|_X & \leq
\Big\|A_s \Big(\sum_{k=1}^N a_k w_{n_k}\Big)\Big\| + \alpha C_1 \varphi_X(1) \Big\| \sum_{k=1}^N a_k w_{n_k}\Big\|_X \\
& \leq \big(C+\alpha \, C_1\varphi_X(1)\big)\Big\| \sum_{k=1}^N a_k w_{n_k}\Big\|_X,
\end{split}\end{equation*}
which shows that $(w_{n_k})$ is a basic sequence in $X$.
\end{proof}

\begin{lemma}\label{lemma:lemmaB}
Let $X$ be an r.i.\ space on $[0,1]$ and $(w_{n_k})$ a
 $q$--lacunary sequence  of Walsh functions.
The following conditions are equivalent.
\begin{enumerate}[(i)]
 \item The operator $T \colon X \to \ell^2$ given by
\begin{equation}\label{eq:operator-T}
Tf:=\big( \langle w_{n_k} ,f \rangle \big)_{k \geq 1}
\end{equation} is continuous.
 \item The continuous embedding $G \subset X'$ holds.
\end{enumerate}\end{lemma}
\begin{proof} 
$\emph{(i)} \Rightarrow \emph{(ii)}$
Assume that $T:X \to \ell^2$ is continuous. 
Then, the adjoint operator $T':\ell^2 \to X'$ is continuous. 
Denote by $(e_k)$ the canonical basis of $\ell^2$. Let $f \in X$. Then,
$$
\langle T'e_k , f \rangle=\langle e_k , Tf \rangle = \langle w_{n_k} , f \rangle,
$$
and so $T'e_k = w_{n_k}$. Hence, for $b=(b_k)_1^\infty \in \ell^2$, we have
$$
T'b=T'\Big( \sum_{k=1}^\infty b_k e_k \Big)=
\sum_{k=1}^\infty b_k \, T'(e_k) =
\sum_{k=1}^\infty b_k w_{n_k}. 
$$
Together with the  continuity of $T'$, it follows that
$$
\Big\|\sum_{k \geq 1} b_{k} w_{n_k}\Big\|_{X'}=\|T'(b_{k}) \|_{X'} \leq \|T'\| \|(b_{k})\|_{\ell^2},
$$
for all $(b_k) \in \ell^2$.
This condition,
as in the proof of Theorem \ref{th:rodin-semenov-walsh},
implies $G \subset X'$.

$\emph{(ii)} \Rightarrow \emph{(i)}$
If  $(w_{n_k})$ is a $q$--lacunary subsequence of Walsh functions and $G \subset X'$, 
we have from Theorem \ref{th:rodin-semenov-walsh} applied to $X'$ that 
there exists a constant $B(X',q)>0$  such that, for $N \geq 1$, we have
$$
\Big\|\sum_{k = 1}^N \langle w_{n_k},f \rangle w_{n_k} \Big\|_{X'} 
\leq B(X',q) \Big( \sum_{k =1}^N \langle w_{n_k},f \rangle^2\Big)^{1/2}.
$$
Let $f \in X$. Fix $N \geq 1$.
Then, from Hölder's inequality for $X$ and $X'$,
\begin{equation*}\begin{split}
\sum_{k = 1}^N \langle w_{n_k},f \rangle^2
& = \int_0^1 f(t) \Big( \sum_{k = 1}^N \langle w_{n_k},f \rangle w_{n_k}\Big)(t) \, dt \\
& \leq \|f\|_X \Big\| \sum_{k = 1}^N \langle w_{n_k},f \rangle w_{n_k}\Big\|_{X'} \\
&  \leq  B(X',q) \|f\|_X \Big( \sum_{k = 1}^N \langle w_{n_k},f \rangle^2\Big)^{1/2}.
\end{split}\end{equation*}
It follows that 
$$
\Big( \sum_{k = 1}^N \langle w_{n_k},f \rangle^2\Big)^{1/2} 
\leq B(X',q) \|f\|_X, \qquad f \in X,
$$ 
that is, $\|Tf\|_{\ell^2} \leq B(X',q) \|f\|_X$, and so $\emph{(i)}$ is established.
\end{proof}

For $n \geq 1$ and any $0 \leq j,k \leq 2^n-1$,
since $w_k$ is constant on the dyadic intervals of order $n$, we have that
$w_k(I^n_j)$ is well--defined and takes values $1$ or $-1$. 
The following property appears in \cite[Theorem 2.b.4]{MR540367}:
$$
w_k(I^n_j)=w_j(I^n_k), \qquad 0 \leq j,k \leq 2^n-1,
$$
but no detail of its proof is given. For the sake of completeness, we give a proof of this fact.

\begin{lemma} \label{lemma:lemmaA}
For $n \geq 1$ and $0 \leq j,k \leq 2^n-1$, we have
$$
w_k(I^n_j)=w_j(I^n_k).
$$
\end{lemma}
\begin{proof} 
The case when either $k=0$ or $j=0$ follows from the fact that
 $w_k(I^n_0)=1$ and $w_0(I^n_k)=1$ for $0 \leq k \leq 2^n-1$.

Fix $n \geq 1$ and $1 \leq j, k \leq 2^n - 1$.
Let
\begin{equation*}\begin{split}
k & =a_1 2^0 + a_2 2^1 + \ldots + a_n 2^{n-1}, \qquad a_1,\ldots,a_n \in \{0,1\}, \\
j & =b_1 2^0 + b_2 2^1 + \ldots + b_n 2^{n-1}, \qquad b_1,\ldots,b_n \in \{0,1\}.
\end{split}\end{equation*}
Fix $x \in I^n_{j}$, and denote by $\{t\}$ the fractional part of $t \in \mathbb{R}$. From
$$
r_i(x)=\sign \sin(2^i \pi x) = r_1(\{2^{i-1} x\}),
$$
we have
\begin{equation}\label{eq:theta}\begin{split}
w_{k}(x) & =(r_1(x))^{a_1} \cdot (r_2(x))^{a_2} \ldots (r_n(x))^{a_n} \\
& =(r_1(x))^{a_1} \cdot (r_1(\{2x\}))^{a_2} \ldots (r_1(\{2^{n-1}x\}))^{a_n}.
\end{split}\end{equation}
Now we use the fact that 
$x \in I^n_{j}$ if and only if 
$$
x=\frac{b_n}{2} + \frac{b_{n-1}}{2^2} + \ldots + \frac{b_1}{2^n} + \varepsilon,
$$
with $0<\varepsilon<1/2^n$. Then, for any $0 \leq i \leq n-1$,
\begin{equation*}\begin{split}
\{2^ix \} & = \Big\{2^i \Big( \frac{b_n}{2} + \frac{b_{n-1}}{2^2} + \ldots + \frac{b_1}{2^n} + \varepsilon \Big)\Big\} \\
& = \Big\{  \frac{b_{n-i}}{2} + \frac{b_{n-i-1}}{2^2} + \ldots + \frac{b_1}{2^{n-i}} + 2^i\varepsilon \Big\}.
\end{split}\end{equation*}
It follows that
$$
r_1(\{2^ix\})=
\left\{\begin{array}{rr}
 1 & \text{if } b_{n-i}=0, \\
 -1 & \text{if } b_{n-i}=1.
\end{array}\right.
$$
Hence,
$$
r_1(\{2^ix\})=(-1)^{b_{n-i}}, \qquad 0 \leq i \leq n-1,
$$
which together with \eqref{eq:theta} implies that
\begin{equation*} \begin{split}
 w_{k}(x) 
 & =(-1)^{a_1 \cdot b_n} \cdot (-1)^{a_2 \cdot b_{n-1}} \ldots (-1)^{a_n \cdot b_1}.
\end{split}\end{equation*}
The symmetry in $a_i$ and $b_i$ 
of the expression above implies that $w_{k}(I^n_{j}) = w_{j}(I^n_{k})$.
\end{proof}

We can now proceed to the proof of the main result.

\begin{proof}[Proof of Theorem \ref{th:complemented-walsh}]

It is clear that $\emph{(ii)} \Rightarrow \emph{(iii)}$.

$\emph{(i)} \Rightarrow \emph{(ii)}$
Let $(w_{n_k})$ be a $q$--lacunary subsequence of Walsh functions.
From the assumption that $G \subset X'$,  Lemma \ref{lemma:lemmaB} implies that
the operator $T:X \to \ell^2$ in \eqref{eq:operator-T} 
is continuous, that is,
$$
\Big(\sum_{k \geq 1} \langle w_{n_k},f \rangle^2 \Big)^{1/2}\leq \|T\| \|f\|_X, \qquad f \in X.
$$
Denote by $P:X \to [w_{n_k}]_X$ the projection
\begin{equation}
 \label{eq:projection}
 f \mapsto Pf:=\sum_{k \geq 1} \langle w_{n_k},f \rangle w_{n_k}.
\end{equation}
From $G \subset X$ and Theorem \ref{th:rodin-semenov-walsh} we
have
$$\|Pf\|_X = \Big\| \sum_{k \geq 1} \langle w_{n_k},f \rangle w_{n_k} \Big\|_X \leq B(X,q)
\Big(\sum_{k \geq 1} \langle w_{n_k},f \rangle^2 \Big)^{1/2}.
$$
It follows that 
$$
\|Pf\|_X \leq B(X,q) \|T\| \|f\|_X, \qquad f \in X,
$$
that is, the subspace $[w_{n_k}]_X$ is complemented in $X$.

$ \emph{(iii)} \Rightarrow \emph{(i)}$ We follow the ideas in the proof 
of the complementability result for Rademacher functions (Theorem \ref{th:complementability})
by Lindenstrauss and Tzafriri \cite[Theorem 2.b.4]{MR540367}.
Assume that $(w_{n_k})$ is a lacunary sequence with $q>1$ such that $[w_{n_k}]_X$
is complemented in $X$. Then, there exists a bounded linear operator $Q:X \to [w_{n_k}]_X$
with $Q^2=Q$. Since $(w_{n_k})$ is a basic sequence, there exists $(q_{n_k}) \subset X^*$ such that 
$$
Qf:=\sum_{k \geq 1} q_{n_k}(f) w_{n_k}, \qquad f \in X.
$$

Let $n \geq 1$ and $X_n$ be the linear subspace generated in $X$ by the characteristic functions 
of the dyadic intervals of order $n$. From the linear independence of the Walsh functions
and from the fact that
$w_k$ is constant on the dyadic intervals of order $n$
for $0 \leq k \leq 2^n-1$, we have that $X_n$ coincides 
with the linear space generated by $\{w_k:k=0,\ldots,2^n-1\}$, that is,
$$
X_n=\big\langle \{\chi_{I^n_k}:k=0,\ldots,2^n-1 \} \big\rangle = \big\langle \{ w_k:k=0,\ldots,2^n-1  \}\big\rangle.
$$

Let $w_{m_1}, \ldots,w_{m_N}$ be
the Walsh functions of the sequence $(w_{n_k})$ with order less or equal than $n$.
Note that $N$ and $m_1, \ldots m_N$ depend on $n$.
Denote by $Q_n$ the restriction of $Q$ to $X_n$, $Q_n:X_n \to X_n$, that is,
\begin{equation}\label{eq:Qn}
Q_n(f):= \sum_{k=1}^N  q_{m_k}( f) w_{m_k}, \qquad f \in X_n.
\end{equation}
For  $0\leq j,k \leq 2^n-1$, denote 
$$
\theta_{j,k}:=w_j(I^n_k).
$$
Let
$T_j:X_n \to X_n$ be the operator defined by
$$
T_j (w_k): =\theta_{j,k} w_k.
$$
Since $X$ is r.i.\ 
and the distribution function of $f$ and $T_j( f)$ coincide for $f \in X_n$,
we have  $\|T_j\|=1$ for $1 \leq j \leq 2^n$. 

Denote by $P_n: X_n \to X_n$ the restriction of the projection in \eqref{eq:projection} to $X_n$,
that is,
$$
P_n(f):= \sum_{k=1}^N \langle w_{m_k},  f   \rangle w_{m_k}, \qquad f \in X_n.
$$ 
We will prove that
\begin{equation}\label{eq:complemented-PQ}
P_n=\frac{1}{2^n} \sum_{j=1}^{2^n} T_jQ_nT_j. 
\end{equation} 

From \eqref{eq:projection},
$$
P_n(w_i)=\sum_{k=1}^N \langle w_{m_k} ,w_i \rangle w_{m_k}.
$$
Since $\langle w_{m_k},w_i\rangle=\delta_{m_k,i}$, it follows that 
$P_n(w_i)=w_i$ for $i \in \{m_1,\ldots,m_N\}$. Otherwise,
$P_n(w_i)=0$.

On the other hand,
\begin{equation*}
 \begin{split}
 \frac{1}{2^n} \sum_{j=1}^{2^n} T_jQ_nT_j(w_i) 
& =\frac{1}{2^n} \sum_{j=1}^{2^n} T_j \Big( \sum_{k=1}^N  q_{m_k}(\theta_{j,i}w_i) w_{m_k} \Big) \\
& =\sum_{k=1}^N q_{m_k}(w_i)\Big( \frac{1}{2^n} \sum_{j=1}^{2^n} \theta_{j,i}\theta_{j,m_k} \Big)w_{m_k} .
 \end{split}
\end{equation*}
Since $Q_n$ is a projection, we have from \eqref{eq:Qn} that $q_{m_k}(w_i)=\delta_{m_k,i}$.  
Thus, \eqref{eq:complemented-PQ} follows from 
$$
\frac{1}{2^n} \sum_{j=1}^{2^n} \theta_{j,i}\theta_{j,m_k}=\delta_{i,m_k}.
$$
Given integers $k,n$ and $m$, we write 
$k=n\oplus m$ whenever the following relation holds:
$$
w_k=w_nw_m.
$$
Then, from Lemma \ref{lemma:lemmaA},
$$
\theta_{j,i}\theta_{j,m_k} = w_j(I^n_i)w_j(I^n_{m_k})=w_i(I^n_j)w_{m_k}(I^n_j)=w_{i \oplus m_k}(I^n_j),
$$
and so it follows that
\begin{equation*}
\begin{split}
\frac{1}{2^n} \sum_{j=1}^{2^n} \theta_{j,i}\theta_{j,m_k} &
=\frac{1}{2^n} \sum_{j=1}^{2^n} w_{i \oplus m_k}(I^n_j) 
=\frac{1}{2^n} \sum_{j=1}^{2^n} \frac{1}{m(I^n_j)} \int_{I^n_j} w_{i\oplus m_k} \, dm\\
&=\int_0^1  w_{i\oplus m_k} \, dm
=\delta_{i,m_k},
\end{split}
\end{equation*}
which proves \eqref{eq:complemented-PQ}.

Let us see that the operators $P_n:X_n \to X_n$ have
 uniformly bounded norm (in $n$). 
From \eqref{eq:complemented-PQ}, since $\|T_j\| = 1$, we have for any $f \in X_n$ that
\begin{equation*} \begin{split}
\|P_nf\|_{X_n} & 
= \Big\|\frac{1}{2^n} \sum_{j=1}^{2^n} T_jQ_nT_jf  \Big\|_{X_n} 
\leq  \frac{1}{2^n} \sum_{j=1}^{2^n} \|T_jQ_nT_jf\|_{X_n} \\
& \leq  \frac{1}{2^n} \sum_{j=1}^{2^n} \|Q_n\| \|f\|_{X_n} 
\leq \|Q\| \|f\|_{X_n},
\end{split}\end{equation*}
that is, $\|P_nf\| \leq \|Q\|$ for all $n \geq 1$.

Next we show that $G \subset X'$ follows from the fact that $P_n:X_n \to X_n$ have
 uniformly bounded norm.
Denote by $T_n:X_n \to \ell^2$ 
the operator given by
$$
T_nf:=(\langle w_{m_k},f\rangle)_{k=1}^{N}, \qquad f \in X_n.
$$
From $X \subset L^1$, there exists a constant $C_1>0$ such 
that $C_1\|f\|_{L^1} \leq  \|f\|_X$ for all $f \in X$. Together with
Theorem \ref{th:sagherzhou4} for $p=1$,
\begin{equation*} \begin{split}
{C_1} {A(1,q)}\Big(\sum_{k=1}^N\langle w_{m_k},f\rangle^2\Big)^{1/2}
& \leq  C_1 \Big\| \sum_{k=1}^{N} \langle w_{m_k},f\rangle w_{m_k} \Big\|_{L^1} \\
& \leq \Big\| \sum_{k=1}^{N} \langle w_{m_k},f\rangle w_{m_k} \Big\|_{X} 
 =\|P_nf\|_{X_n} \\
& \leq \|Q\| \|f\|_X.
\end{split}\end{equation*}
It follows, for some constant $C_2 >0$, that 
$$
\|T_nf\|_2 \leq C_2 \|f\|_{X_n}, \qquad f \in X_n,
$$
for all  $n \geq 1$. Thus, the adjoint operator
$$
T'_n: \ell^2 \to X'_n
$$
is bounded with $\|T'_n\| \leq C_2$, $n \geq 1$.
Since, for any $(b_k) \in \ell^2$, 
$$
T'_n(b_k)=\sum_{k=1}^{N} b_{k} w_{m_k},
$$
we have,
as in the proof of Lemma \ref{lemma:lemmaB},   that
$$
\Big\| \sum_{k=1}^{N} b_{k} w_{m_k} \Big\|_{X'}  \leq \|T'_n\|
\Big( \sum_{k=1}^{N} b_k^2 \Big)^{1/2}
\leq C_2 \Big(\sum_{k=1}^{N} b_k^2 \Big)^{1/2}. 
$$
Since this inequality holds for all $n\geq 1$, we have 
$$
\Big\| \sum_{k \geq 1} b_{k} w_{n_k} \Big\|_{X'}  
\leq C_2 \Big(\sum_{k\geq 1} b_k^2 \Big)^{1/2}.
$$
This inequality, together with the embedding $X' \subset L^1$ and Theorem \ref{th:sagherzhou4},
shows that the subspace $[w_{n_k}]_{X'}$ 
is isomorphic to $\ell^2$. From Theorem \ref{th:rodin-semenov-walsh}, 
this is equivalent to $G \subset X'$ .

Now we show that $G \subset X$. Let $n \geq 1$. For $f \in X_n$ and $g \in X_n'$,
\begin{equation*}\begin{split}
\int_0^1 P_n(f) \, g \,dm
& = \int_0^1 \Big( \sum_{k=1}^{N} \langle w_{m_k},f \rangle w_{m_k} \Big) g \,dm \\
& = \sum_{k=1}^{N} \langle w_{m_k},f \rangle \langle w_{m_k}, g\rangle  \\
& = \int_0^1 f \, P_n(g) \,dm.
\end{split}\end{equation*}
Thus,
$$
\Big| \int_0^1 f \, P_n(g) \,dm \Big| =
\Big| \int_0^1 P_n(f) \,  g \,dm \Big|\leq \|P_n(f)\|_X \|g\|_{X'} \leq \|Q\| \|f\|_X \|g\|_{X'}.
$$
It follows that
\begin{equation*}\begin{split}
\|P_ng\|_{X'_n} 
& = \sup \Big\{ \Big|\int_0^1 f \,  P_n(g) \,dm \Big| : \|f\|_{X_n}\leq 1 \Big\} \\
& \leq \|Q\| \|g\|_{X'_n}.
\end{split}\end{equation*}

The argument above, together with $X'_n$ in place of $X_n$, shows that $G \subset X''$
follows from the fact that $P_n: X'_n \to X'_n$ is uniformly bounded. 
Thus, for the separable parts we have that
$G_0 \subset (X'')_0$.  
Since $(X'')_0=X_0$ it follows that
$$
G = G_0 \subset (X'')_0 = X_0 \subset X,
$$
which concludes the proof.

\end{proof}
%%%%%%%%%%%%%%%%%%%%%%%%%%%%%%%%%%%%%%%%%%%%%%

\paragraph{\textbf{Acknowledgements.}} 
This work is part of the Ph.D. thesis of the author which 
has been written
at University of Sevilla under the supervision of Prof.\ G.\ P.\ Curbera.

%%%%%%%%%%%%%%%%%%%%%%%%%%%%%%%%%%%%%%%%%%%%%

%%%%%%%%%%%%%%%%%%%%%%%%%%%%%%%%%%%%%%%%%%%%%

%%%%%%%%%%%%%%%%%%%%
%%% BIBLIOGRAPHY %%%
%%%%%%%%%%%%%%%%%%%%

\bibliography{bibliography}
\bibliographystyle{plain}

\end{document}